\documentclass[10pt,DIV=11,a4paper]{scrartcl}

\usepackage{scrlayer-scrpage}
\lohead{Dietert and Hirsch}
\cohead{Rough hypoelliptic operator}
\rohead{Page \pagemark}
\cofoot{}
\usepackage[pdfusetitle,hidelinks]{hyperref}

\usepackage{amsmath}
\usepackage{amsthm}
\usepackage{amssymb}
\usepackage{mathtools}
\usepackage{newtxmath}

\usepackage{color}

\usepackage{enumitem}
\setlist{noitemsep}

\usepackage[capitalise]{cleveref}
\usepackage[backend=biber,isbn=false,maxcitenames=5,maxnames=5,sortcites]{biblatex}
\DeclareSourcemap{
  \maps[datatype=bibtex]{
    \map{
      \step[fieldsource=doi,final]
      \step[fieldset=url,null]
      \step[fieldset=urldate,null]
    }
  }
}

\bibliography{lit.bib}

\usepackage{tikz}
\usetikzlibrary{arrows.meta}
\usetikzlibrary{shapes.symbols}
\usepackage{pgfplots}
\pgfplotsset{compat=1.14}
\pgfplotsset{
    tick align=outside,
    x grid style={white},
    xmajorgrids,
    y grid style={white},
    ymajorgrids,
    axis line style={white},
    axis background/.style={fill=white!92!black},
    legend style={draw=white, fill=white},
    legend cell align={left}
}

\newtheorem{thm}{Theorem}

\newtheorem{lemma}[thm]{Lemma}

\newtheorem{definition}[thm]{Definition}

\newtheorem{hypothesis}{Hypothesis}
\theoremstyle{remark}
\newtheorem{remark}{Remark}

\newcommand{\hypref}[1]{\textbf{(H\ref{#1})}}

\newcommand{\tnorm}[1]{{\left\vert\kern-0.25ex\left\vert\kern-0.25ex\left\vert #1
    \right\vert\kern-0.25ex\right\vert\kern-0.25ex\right\vert}}
\newcommand{\R}{\mathbb{R}}
\newcommand{\N}{\mathbb{N}}

\newcommand{\fsL}{\textnormal{L}} % Lp spaces
\newcommand{\fsH}{\textnormal{H}} % Sobolev spaces
\newcommand{\fsC}{\mathscr{C}} % C For function spaces
\newcommand{\fsHyp}{\textnormal{H}_{\textnormal{hyp}}} % H_hyp space

\newcommand{\ee}{\mathrm{e}} % constant e for exponential function
\newcommand{\dd}{\mathrm{d}} % constant d for differential

\newcommand{\id}{\mathrm{id}} % Identity matrix
\newcommand{\init}{\mathrm{in}} % initial data
\newcommand{\ext}{\mathrm{ext}} % Extension

\newcommand{\conv}{*} % Convolution
\newcommand{\ind}{\vmathbb{1}} % Indicator function
\DeclareMathOperator{\supp}{supp}
\DeclareMathOperator{\divergence}{div}

\newcommand{\oP}{P} % Operator P
\newcommand{\oL}{L} % Principal part operator
\newcommand{\cyl}{C^{X_0}} % Cylinder

\title{Regularity for rough hypoelliptic equations}
\author{Helge Dietert \and Jonas Hirsch}
\begin{document}
\maketitle
\begin{abstract}
  We present a general approach to obtain a weak Harnack inequality for
  rough hypoellipitic equations, e.g.\ kinetic equations. The proof is
  constructive and does not study the commutator structure but rather
  compares the rough solution with a smooth problem for which the
  estimates are assumed.
\end{abstract}

\section{Introduction}

\subsection{Motivation}

One motivation is kinetic theory describing a density
\(f=f(t,x,v)\) at a time~\(t\) over the phase space consisting of a
spatial position~\(x\) and a velocity~\(v\). For a collisional
evolution, like the Boltzmann or Landau equation, the evolution is
then given by
\begin{equation}
  \label{eq:boltzmann}
  \partial_t f + v \cdot \nabla_x f = Q(f),
\end{equation}
where \(Q\) is a collision operator. In the most basic form, \(Q\) is
a diffusion operator in the velocity variable \(v\) so that we arrive
at
\begin{equation}
  \label{eq:kinetic-smooth}
  \partial_t f + v \cdot \nabla_x f = \nabla_v \cdot \nabla_v f.
\end{equation}

The evolution~\eqref{eq:kinetic-smooth} is not parabolic because there
is no diffusion in the spatial position \(x\). The fundamental
solution, computed explicitly by
\textcite{kolmogoroff-1934-zufaellige-bewegungen} in 1934, shows that,
nevertheless, a solution is smooth in all directions.

In a general setting, \textcite{hoermander-1967-hypoelliptic}
understood in 1967 this smoothing property. For smooth vector fields
\(X_0,X_1,\dots,X_m\), he looked at solutions to the equation
\begin{equation}
  \label{eq:hypoelliptic-setup}
  X_0 u + \sum_{i=1}^{m} (-X_i)^* X_i u = S
\end{equation}
and called the equation \emph{hypoelliptic} if the smoothness of
\(S\) implies that \(u\) is smooth. He then shows that
\eqref{eq:hypoelliptic-setup} is hypoelliptic if \(X_0,X_1,\dots,X_m\)
and their commutators span the full space at every point.

A different development was the question of regularity for elliptic
equations with \emph{rough coefficients}. Such a regularity was proved
by \textcite{giorgi-1957-sulla} in 1957 and
\textcite{nash-1958-continuity} in 1958, also covering the parabolic
case.

The combination of these ideas saw a lot of recent interest
\cite{wang-zhang-2009-ultraparabolic,wang-zhang-2011-ultraparabolic,golse-imbert-mouhot-vasseur-2019-harnack-inequality,guerand-imbert-2021-preprint-log-transform,guerand-mouhot-2021-preprint-quantitative-de-giorgi,anceschi-rebucci-2021-preprint-weak-regularity-kolmogorov,anceschi-polidoro-ragusa-2019-mosers-kolmogorov,zhu-2020-preprint-velocity-averaging,garain-nyström-2022-preprint-kolmogorov-fokker-planck}
as it is a path for regularity results for nonlinear kinetic
equations, where the solution satisfies schematically
\begin{equation}
  \label{eq:kinetic-rough}
  \partial_t f + v \cdot \nabla_x f = \nabla_v \cdot(a \nabla_v f)
\end{equation}
and \(a\) depends again on \(f\). On this level, we do not know any
regularity on \(f\) so that we just assume that \(a\) is bounded from
below and above, i.e.\ \(a\) is a rough coefficient. If we can still
obtain a regularity result, we can use it to bootstrap regularity as
explained in
\cite{imbert-mouhot-2021-schauder-toy-model,imbert-silvestre-2020-harnack-boltzmann,anceschi-zhu-2021-preprint-fokker-planck}.

A related direction is the study of sub-Riemannian geometry which asks
similar questions without a drift \(X_0\). The lack of the drift seems
to simplify several problems and quite general results are
available. Along this direction we refer to
\cite{capogna-citti-rea-2013-subelliptic-analogue-aronson-serrin-harnack-inequality}
as a starting point.

\subsection{General setting and main results}

Our observation is that the smoothing property of an hypoelliptic
operator implies in an robust way the key steps of regularity for a
rough version: the supremum bound and the weak Harnack inequality.

In this general setting, we study functions \(u=u(t,x)\) where
\(t\in\R\) is a special time-variable and \(x\in\R^n\) is a general
space. Then suppose smooth vector-fields \(\tilde X_0,X_1,\dots,X_m\)
acting only along the spatial directions \(x\), i.e.\
\(X_i = \sum_{j=1}^{n} X^j_i(t,x) \partial_{x_j}\) with smooth
coefficients \((X^j_i)_{i,j}\). Using the standard \(\fsL^2(\R^n)\)
define the adjoints \(X_i^*\) of \(X_i\) and let
\begin{equation*}
  X_i^t = - X_i^* \qquad \text{for } i=1,\dots,m.
\end{equation*}
We consider the smooth operator
\begin{equation}
  \label{eq:def-p0}
  X_0 - \oL_0
  \quad\text{where}\quad
  X_0 = \partial_t + \tilde X_0
  \text{ and }
  \oL_0 := \sum_{i=1}^{m} X_i^t X_i.
\end{equation}
The natural functional spaces for solutions has already been
identified in \textcite{hoermander-1967-hypoelliptic}, see also
\cite{albritton-armstrong-mourrat-novack-2019-preprint-variational-fokker-planck,anceschi-rebucci-2021-preprint-weak-regularity-kolmogorov}.
We introduce the space \(\fsHyp^1\) defined by the norm
\begin{equation}
  \label{eq:definition-hyp}
  \| u \|_{\fsHyp^1} = \| u \|_{\fsL^2} + \| \vec{X} u \|_{\fsL^2}
\end{equation}
where $\vec{X}u=(X_1u, \dotsc, X_mu)$ and we denote by \(\fsHyp^{-1}\)
the dual space of \(\fsHyp^{1}\). Throughout we will consider
classical weak solutions \(u\) with \(u \in \fsHyp^1\) and
\(X_0 u \in \fsHyp^{-1}\).

For a point \(x_0 \in \R^n\), let \(B_r(x_0) \subset \R^n\) be the
standard open Euclidean ball of radius \(r\). For a parabolic cylinder, we
include the drift \(X_0\).
\begin{definition}[Parabolic cylinder \(\cyl\)]
  \label{def:parabolic-cyl}
  For a point \((t_0,x_0) \in \R^{1+n}\) solve the transport equation
  \begin{equation}
    \left\{
      \begin{aligned}
        & X_0 \eta = 0
        & &\text{in } \R^{1+n}, \\
        &w(t_0,\cdot) = \ind_{B_r(x_0)}
        & &\text{on } \{t_0 \} \times \R^n.
      \end{aligned}
    \right.
  \end{equation}
  The parabolic cylinder \(\cyl_{s,r}(t_0,x_0) \subset \R^{1+n}\) with
  time size \(s\) and space size \(r\) is
  \begin{equation}
    \label{eq:def-parabolic-cyl}
    \cyl_{s,r}(t_0,x_0) = \supp \eta \cap (t_0-s,t_0].
  \end{equation}
\end{definition}
\begin{figure}[h]
  \centering
  \begin{tikzpicture}
    % Cylinder
    \draw[thick,fill=green!20] (1,1) -- (3,1) -- (2.5,-1) -- (0.5,-1)
    -- cycle;
    \draw[-{Stealth[scale=2]}] (2.75,0) -> (2.875,0.5)
    node[anchor=west] {\(\tilde X_0\)};
    % Coordinate system
    \draw[->] (0,-1.5) -- (0,1.5) node[anchor=east] {\(t\)};
    \draw[->] (-0.1,0) -- (5,0) node[anchor=north west]
    {\(x\in\R^n\)};
    \draw[fill=black] (2,1) circle (0.05) node[anchor=south]
    {\((t_0,x_0)\)};
    \draw[thick, dotted] (1,1) -- (-0.1,1) node[anchor=east] {\(t_0\)};
    \draw[thick, dotted] (0.5,-1) -- (-0.1,-1) node[anchor=east] {\(t_0-s\)};
  \end{tikzpicture}
  \caption{Illustration of \cref{def:parabolic-cyl} of a parabolic
    cylinder \(\cyl_{s,r}(t_0,x_0)\).}
  \label{fig:parabolic-cyl}
\end{figure}
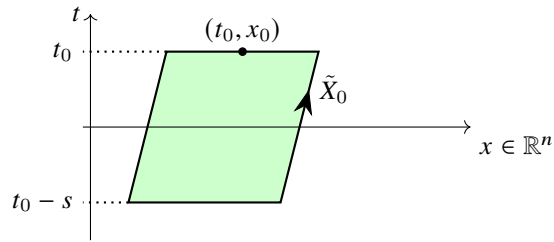
\begin{remark}
  If \(\tilde X_0\) is independent of time, the transport equation is
  solved by the semigroup \(\ee^{t\tilde X_0}\) and we find
  \begin{equation*}
    \cyl_{s,r}(t_0,x_0) =
    \{
    (t,y) \in (t_0-s,t_0] :
    y \in \ee^{(t_0-t)\tilde X_0} B_r(x_0)
    \}.
  \end{equation*}
\end{remark}

We then capture the hypoelliptic behaviour of \(X_0 - \oL_0\) by
supposing estimates gaining local integrability.
\begin{hypothesis}\label{h:smooth-ivp}
  Suppose a parabolic domain
  \(\Omega_t = (t_1,t_2] \times B \subset \R^{1+n}\) for times
  \(-\infty<t_1<t_2<\infty\) and a bounded ball \(B \subset \R^n\) and
  suppose an extended domain
  \(\Omega^\ext_t = (t_1,t_2] \times B^\ext\) for a ball
  \(B \subset B^\ext\) (with the possibility \(B^\ext = \R^n\)) and an
  extended (possible degenerate) elliptic operator \(\oL^\ext_0\) in
  divergence form with vanishing lower order terms such that
  \(\oL^\ext_0 = \oL_0\) on \(\Omega_t\). Then suppose integrabilities
  \(p_1>2\) and \(\gamma_0\le\gamma_1\le 2\) and a constant \(C_0>0\)
  such that for any time \(t_\init \in (t_1,t_2)\) and functions
  \(G \in \fsL^{\gamma_0}(\Omega^\ext_t), F = (F_1,\dots,F_m) \in
  \fsL^{\gamma_0}(\Omega^\ext_t,\R^m)\) with
  \(\supp G \cup \supp F \subset \Omega_t \cap \{t \ge t_\init\}\)
  there exists a function
  \(w : \Omega_t^\ext \cap \{t \ge t_\init\} \to \R\) satisfying
  \begin{equation}
    \label{eq:l0-problem}
    \left\{
      \begin{aligned}
        &(X_0 - \oL_0^\ext) w \ge G + \sum_{i=1}^{m} X_i^t F^i
        & &\text{in } \Omega_t^\ext \cap \{t>t_\init\}, \\
        &w \ge 0
        & &\text{on } \{t_\init\} \times B^\ext \cup (t_\init,t_2) \times
        \partial B^\ext
      \end{aligned}
    \right.
  \end{equation}
  and
  \begin{equation}
    \label{eq:l0-estimate}
    \| w \|_{\fsL^{p_1}(\Omega_t \cap \{t > t_\init\})}
    \le C_0\,
    \left(
      \| G \|_{\fsL^{\gamma_0}(\Omega_t \cap \{t > t_\init\})}
      +
      \| F \|_{\fsL^{\gamma_1}(\Omega_t \cap \{t > t_\init\})}
    \right).
  \end{equation}
\end{hypothesis}
\begin{remark}
  In the case of kinetic or Kolmogorov equations there exists a
  fundamental solution of \(X_0 - \oL_0\) over the whole space and we
  can obtain the sought \(w\) and the estimates by the fundamental
  solution with \(\oL^\ext_0 = \oL_0\) and \(B^\ext = \R^n\).

  If we only have local estimates for solutions of \(X_0 - \oL_0\),
  then it is difficult to construct a solution with boundary condition
  \(w \ge 0\) as it is not clear due to the degeneracy of \(\oL_0\)
  what boundary conditions can be imposed. Therefore, we allow a
  different extension \(\oL_0^\ext\) which we can take as
  \(\oL_0^\ext = \oL_0 + \nabla \cdot ((1-\chi)^2 \nabla\cdot)\) with a
  cutoff \(\chi\) and the normal gradient \(\nabla\) on \(\R^n\). This
  then allows the same local estimates and the imposition of boundary
  condition \(w=0\) on
  \(\{t_\init\} \times B^\ext \cup (t_\init,t_2) \times \partial
  B^\ext\).
\end{remark}

In this setting we study the differential operator \(\oP\) with rough
coefficients defined by
\begin{equation}
  \label{eq:def-p}
  \oP u := X_0 u
  - \sum_{i=1}^m X_i^t A^i(t, x, u, \vec{X}u) - B(t,x,u,\vec{X}u)
\end{equation}
where
\begin{align}
  \label{eq:def-a}
  A^i(t,x,z,p)= a^{ij}(t, x) p_j + b^i(t, x) z - f^i(t, x),\\
  \label{eq:def-b}
  B(t,x,z,p)=c^i(t, x)p_i + d(t, x) z - g(t, x).
\end{align}
For the diffusion coefficient assume the uniform lower bound
\(\lambda\) on the symmetric part
\begin{equation}
  \label{eq:bound-a-lower}
  \lambda\, \id \le \left(\frac{a^{ij}+a^{ji}}{2}\right)_{ij}
  \text{ in the sense of matrices}
\end{equation}
and assume that the coefficients are bounded by a function
\(\Lambda=\Lambda(t,x)\) as
\begin{equation}
  \label{eq:bound-a-upper}
  |a^{ij}(t,x)| \le \frac{\Lambda(t,x)}{n}
  \text{ for all } i,j=1,\dots,m.
\end{equation}

Our first result is a supremum bound for subsolutions.
\begin{thm}[Supremum bound]\label{thm:supremum-bound}
  Assume a parabolic cylinder \(\cyl_{S,R}(t_0,x_0)\) around a point
  \((t_0,x_0) \in \R^{1+n}\) with \(0<S\) and \(0<R\) containing the
  cylinder \(\cyl_{s,r}(t_0,x_0)\) with \(0<s<S\) and \(0<r<R\) and
  assume \hypref{h:smooth-ivp} is satisfied for the smooth problem on
  a domain \(\Omega_t\) containing the closure of
  \(\cyl_{S,R}(t_0,x_0)\).  Take \(2<p_0<p_1\) and integrabilities
  \(q_\Lambda,q_b,q_c,q_d\) satisfying
  \begin{align*}
    \frac{1}{q_\Lambda} &\le \min\left\{\frac 12 - \frac{1}{p_0},\frac{1}{\gamma_1} - \frac 12\right\},&
    \frac{1}{q_b} &\le \min\left\{\frac 12 \left(\frac{1}{\gamma_0} -
                    \frac{1}{p_0}\right),\frac 12 - \frac{1}{p_0}\right\},\\
    \frac{1}{q_c} &\le \min\left\{\frac{1}{\gamma_0} - \frac 12,
                    \frac 12 - \frac{1}{p_0}\right\},&
    \frac{1}{q_d} &\le \min\left\{\frac{1}{\gamma_0} - \frac{1}{p_0},
                    1 - \frac{2}{p_0}\right\}.
  \end{align*}

  Then there exist constants \(C_S,\beta>0\)
  such that a function \(u\) satisfying \(\oP u \le 0\) on
  \(\cyl_{S,R}(t_0,x_0)\) for a differential operator \(\oP\) of
  \eqref{eq:def-p} is bounded as
  \begin{equation*}
    \sup_{\cyl_{s,r}(t_0,x_0)} u
    \le C_S\, (1+\delta_S)^\beta
    \big(
    \| u \|_{\fsL^{1}(\cyl_{S,R}(t_0,x_0))}
      +
    \| f \|_{\fsL^{q_b}(\cyl_{S,R}(t_0,x_0))}
    +
    \| g \|_{\fsL^{q_d}(\cyl_{S,R}(t_0,x_0))}
    \big)
  \end{equation*}
  where
  \begin{equation*}
    \delta_S =
    \| \Lambda \|_{\fsL^{q_\Lambda}(\cyl_{S,R}(t_0,x_0))}
    +
    \| b \|_{\fsL^{q_b}(\cyl_{S,R}(t_0,x_0))}
    +
    \| c \|_{\fsL^{q_c}(\cyl_{S,R}(t_0,x_0))}
    +
    \| d \|_{\fsL^{q_d}(\cyl_{S,R}(t_0,x_0))}.
  \end{equation*}
\end{thm}

The next step for the regularity of solutions to the rough operator
\(\oP\) is a weak Harnack inequality. For a nonnegative solution \(u\)
and a cylinder \(\cyl_{1,1}(0,x_0)\) we want to conclude that \(u\) is
strictly positive in \(\cyl_{1,1}(0,x_0) \cap \{-1/3 \le t \le 0\}\)
if \(\{u \ge 1\}\) is a set of positive measure in
\(\cyl_{1,1}(0,x_0) \cap \{-1/3 \le t \le 0\}\) beforehand. The idea
is to again use a similar property for the smooth dual problem. For
the conclusion, we need a larger domain with an arbitrary smooth
cutoff. This is captured in the following hypothesis, cf.\
\cref{fig:sketch-small-cutoff}.

\begin{hypothesis}\label{h:small-cutoff}
  From the point \((0,x_0) \in \R^{1+n}\), there exists for every
  \(R>1\) bounded domains
  \(\cyl_{1,2}(0,x_0) \subset \tilde \Sigma_R \subset \Sigma_R \subset
  [-1,0] \times \R^n\) and smooth cutoffs
  \(\tilde \eta_R,\eta_r : [-1,0] \times \R^n \to [0,1]\) with
  \(\supp \tilde \eta_R \subset \tilde \Sigma_R\) and
  \(\supp \eta_R \subset \Sigma_R\) and \(\tilde \eta_R \equiv 1\) on
  \(\cyl_{1,1}(0,x_0)\) and \(\eta_R \equiv 1\) on
  \(\tilde \Sigma_r\) such that
  \begin{equation*}
    X_0 \tilde \eta_R = X_0 \eta_R = 0
  \end{equation*}
  and
  \begin{equation*}
    \| \vec{X} \tilde \eta_R \|_{\fsL^\infty} \le 1,\quad
    \| \vec{X}\vec{X} \tilde \eta_R \|_{\fsL^\infty} \le \frac 1R,\quad
    \| \vec{X} \eta_R \|_{\fsL^\infty} \le 1.
  \end{equation*}
\end{hypothesis}

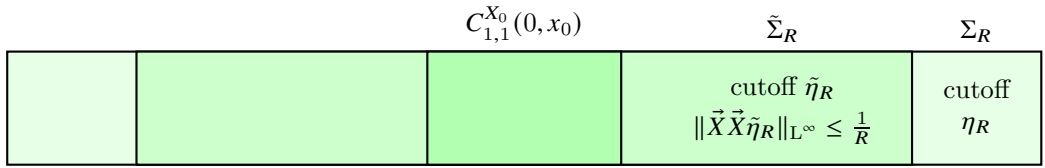
\begin{figure}[thb]
  \centering
  \begin{tikzpicture}[xscale=1.7,yscale=1.5]
    \draw[thick,fill=green!10] (-4,0) -- (4,0) -- (4,-1) -- (-4,-1)
    -- cycle;
    \draw[thick,fill=green!20] (-3,0) -- (3,0) -- (3,-1) -- (-3,-1)
    -- cycle;
    \draw[thick,fill=green!30] (-0.75,0) -- (0.75,0) -- (0.75,-1) -- (-0.75,-1)
    -- cycle;
    \node[anchor=south] at (0,0) {\(\cyl_{1,1}(0,x_0)\)};
    \node[anchor=south] at (2,0) {\(\tilde \Sigma_R\)};
    \node[anchor=south] at (3.5,0) {\(\Sigma_R\)};
    \node at (2,-0.33) {cutoff \(\tilde \eta_R\)};
    \node at (2,-0.66) {\(\| \vec X \vec X \tilde
      \eta_R\|_{\fsL^\infty} \le \frac 1R\)};
    \node at (3.5,-0.33) {cutoff};
    \node at (3.5,-0.66) {\(\eta_R\)};
  \end{tikzpicture}
  \caption{Illustration of the enlarged domain with the control of
    cutoff in \hypref{h:small-cutoff}.}
  \label{fig:sketch-small-cutoff}
\end{figure}

\begin{remark}
  In simple hypoelliptic cases like in kinetic theory, the sets
  \(\tilde \Sigma_R\) and \(\Sigma_R\) can be taken as parabolic
  cylinders \(\cyl_{\bar R,1}(0,x_0)\) for large enough \(\bar R\) and
  \(\tilde \eta, \eta\) can be taken as solutions of
  \(X_0 \tilde \eta = X_0 \eta = 0\) with a prescribed standard cutoff
  at \(\{t=0\}\). For general operators, it can be assumed locally by
  using the underlying scaling of the vector fields
  \(X_0,X_1,\dots,X_m\).
\end{remark}

We now state the assumption of the smooth dual problem assuming
\hypref{h:small-cutoff}, cf. \cref{fig:dual-spreading}.
\begin{hypothesis}\label{h:dual-spreading}
  For the point \(x_0 \in \R^n\), assume constants \(\eta,\mu_0>0\)
  such that the problem
  \begin{equation}
    \label{eq:weak-harnack-dual}
    \left\{
      \begin{aligned}
        &(X_0^t-\oL_0)w = \ind_{E}
        & &\text{in } (-1,0] \times \R^n, \\
        &w(-1,\cdot) = 0
        & &\text{on } \{t=-1\} \times \R^n
      \end{aligned}
    \right.
  \end{equation}
  for a set \(E \subset \cyl_{1,1}(0,x_0) \cap \{t \le -2/3\}\) with
  \(|E| \ge \eta |\cyl_{1,1}(0,x_0) \cap \{t \le -2/3\}|\) has a
  solution \(w \ge 0\) satisfying
  \begin{equation}
    \label{eq:weak-harnack-smooth}
    w(t,x) \ge \mu_0 \quad\text{for}\quad (t,x) \in \cyl_{1/2,2}(0,x_0)
  \end{equation}
  and
  \begin{equation*}
    \| w \|_{\fsL^1((-1,0)\times \R^n)} \lesssim 1.
  \end{equation*}
  Assume further an integrability \(p_2 \ge 2\). Then for any \(R>0\),
  there exists a constant \(C_d(R)\) such that
  \begin{equation*}
    \| w \|_{\fsL^{p_2}(\Sigma_R)}
    + \| \vec{X} w \|_{\fsL^{p_2}(\Sigma_R)}
    \le C_d(R).
  \end{equation*}
\end{hypothesis}

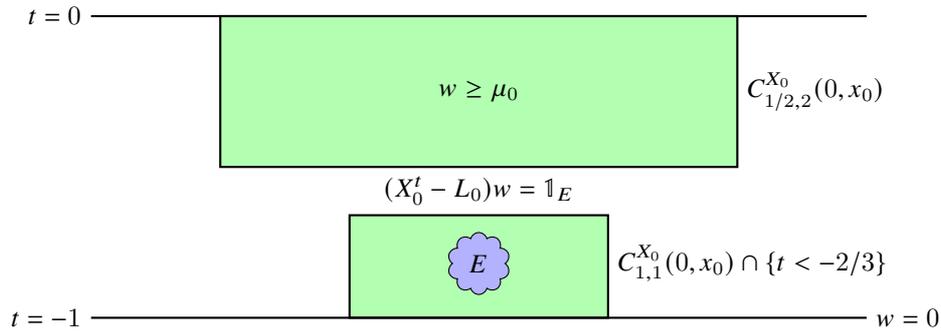
\begin{figure}[htb]
  \centering
  \begin{tikzpicture}[xscale=1.7,yscale=4]
    \draw[thick,fill=green!30] (-2,0) -- (2,0) -- (2,-0.5) -- (-2,-0.5)
    -- cycle;
    \draw[thick,fill=green!30] (-1,-0.66) -- (1,-0.66) -- (1,-1) -- (-1,-1)
    -- cycle;
    \draw[thick] (3,0) -- (-3,0) node[anchor=east] {\(t=0\)};
    \draw[thick] (3,-1) -- (-3,-1) node[anchor=east] {\(t=-1\)};
    \node[cloud, fill=blue!30, draw=black] at (0,-0.82){\(E\)};
    \node[anchor=west] at (2,-0.25) {\(\cyl_{1/2,2}(0,x_0)\)};
    \node[anchor=west] at (1,-0.82) {\(\cyl_{1,1}(0,x_0) \cap
      \{t<-2/3\}\)};
    \node[anchor=north] at (0,-0.5) {\((X_0^t-\oL_0) w = \ind_E\)};
    \node[anchor=west] at (3,-1) {\(w=0\)};
    \node at (0,-0.25) {\(w \ge \mu_0\)};
  \end{tikzpicture}
  \caption{Illustration of \hypref{h:dual-spreading}.}
  \label{fig:dual-spreading}
\end{figure}

We then satisfy a weak Harnack inequality for the rough problem, cf.\
\cref{fig:weak-harnack}.
\begin{thm}[Weak Harnack inequality]\label{thm:weak-harnack}

  Assume that \cref{thm:supremum-bound} applies to
  \(\cyl_{1/2,2}(0,x_0)\) and additionally \hypref{h:small-cutoff}
  and \hypref{h:dual-spreading}. Then there exists \(C_R\) such that
  for \(\delta_S, \Delta>0\) there exists
  \(\mu,\epsilon>0\) such that with \(R = C_R (1+\delta_S)^{\beta}\)
  a nonnegative supersolution \(u\ge 0\) of
  \begin{equation*}
    \oP u \ge 0 \text{ over } \Sigma_R
  \end{equation*}
  with
  \begin{equation}
    \label{eq:weak-harnack-e}
    |E| \ge \eta |\cyl_{1,1}(0,x_0) \cap \{t \le -2/3\}|
    \text{ for }
    E = \{u \ge 1\} \cap \cyl_{1,1}(0,x_0) \cap \{t \le -2/3\}
  \end{equation}
  satisfies
  \begin{equation*}
    u(t,x) \ge \mu \text{ in } \cyl_{1/3,1}(0,x_0)
  \end{equation*}
  if \(\oP\) is an operator of the form~\eqref{eq:def-p} with
  \begin{equation*}
    \| \Lambda \|_{\fsL^{q_\Lambda}(\cyl_{\frac 12,2}(t_0,x_0))}
    +
    \| c \|_{\fsL^{q_c}(\cyl_{\frac 12,2}(t_0,x_0))}
    \le \delta_S
  \end{equation*}
  and
  \begin{equation*}
    \| \Lambda \|_{\fsL^{\bar q_2}(\Sigma_R)}
    + \| \tilde c \|_{\fsL^{\bar q_2}(\Sigma_R)}
    \le \Delta
    \qquad
    \text{with }
    \frac{1}{\bar q_2} + \frac{1}{p_2} = \frac 12
  \end{equation*}
  and
  \begin{equation*}
    \| f - ub \|_{\fsL^{q_b}(\Sigma_R)} + \| g - ud \|_{\fsL^{q_d}(\Sigma_R)}
    \le \epsilon.
  \end{equation*}
\end{thm}

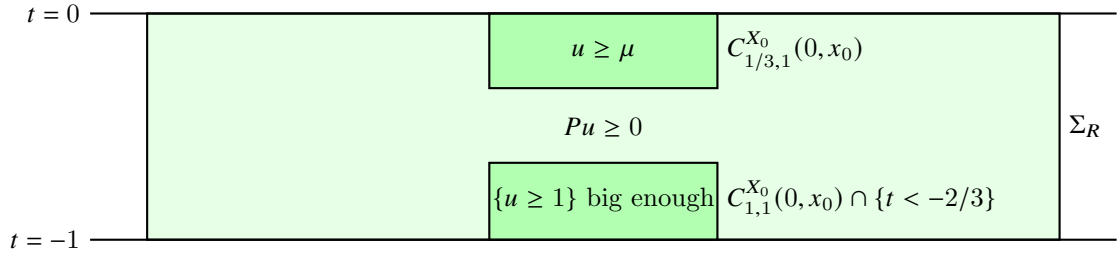
\begin{figure}[htb]
  \centering
  \begin{tikzpicture}[xscale=1.5,yscale=3]
    \draw[thick,fill=green!10] (-4,0) -- (4,0) -- (4,-1) -- (-4,-1) --
    cycle; \draw[thick,fill=green!30] (-1,0) -- (1,0) -- (1,-0.33) --
    (-1,-0.33) -- cycle; \draw[thick,fill=green!30] (-1,-0.66) --
    (1,-0.66) -- (1,-1) -- (-1,-1) -- cycle; \draw[thick] (4.5,0) --
    (-4.5,0) node[anchor=east] {\(t=0\)}; \draw[thick] (4.5,-1) --
    (-4.5,-1) node[anchor=east] {\(t=-1\)}; \node[anchor=west] at
    (1,-0.166) {\(\cyl_{1/3,1}(0,x_0)\)}; \node[anchor=west] at
    (1,-0.82) {\(\cyl_{1,1}(0,x_0) \cap \{t<-2/3\}\)};
    \node[anchor=west] at (4,-0.5) {\(\Sigma_R\)}; \node at (0,-0.166)
    {\(u \ge \mu\)}; \node at (0,-0.5) {\(\oP u \ge 0\)}; \node at
    (0,-0.82) {\(\{u \ge 1\}\) big enough};
  \end{tikzpicture}
  \caption{Illustration of weak Harnack inequality
    (\cref{thm:weak-harnack}).}
  \label{fig:weak-harnack}
\end{figure}

After some preliminary remarks in \cref{sec:preliminaries}, we will
prove \cref{thm:supremum-bound} in \cref{sec:l-inf} and
\cref{thm:weak-harnack} in \cref{sec:harnack}. Both proofs are
quantitative. The proofs are presented as a priori estimates and we
briefly discuss the required function space in
\cref{sec:function-spaces}.

\subsection{Application to hypoelliptic operator}

In the first work of hypoellipticity by
\textcite{hoermander-1967-hypoelliptic}, the key estimate is that a
solution \(u\) to \((X_0-\oL_0)u = G + X_i^t F_i\) satisfies
\(\| u \|_{\fsH^s} \lesssim \| u \|_2 + \| F \|_2 + \| G \|_2\) for
some \(s>0\) under the commutator condition. This then shows the
estimate \hypref{h:smooth-ivp} by Sobolev embedding with
\(\gamma_0=\gamma_1=2\), see \cref{sec:construction-comparison}.

In this general setting, \textcite{bony-1969-principe-harnack-cauchy}
proved a strong maximum principle which yields the claimed spreading
of positivity in \hypref{h:dual-spreading} by a compactness argument.

In the kinetic or general Kolmogorov setting, there is an explicit
fundamental solution from which all estimates on the smooth problem
can be easily verified. Using the best possible integrabilities, the
assumed integrabilities on the lower order terms are as expected from
the classical parabolic case arbitrary close to the integrabilities
expected from scaling. For the upper bound \(\Lambda\) on the
diffusion coefficients \(a^{ij}\), our result matches
\textcite{trudinger-1971} in the classical case.

The kinetic or Kolmogorov equation have an underlying scaling and
group structure (corresponding to Galilean transformation in the
kinetic theory) which allows to conclude from the weak Harnack result
(\cref{thm:weak-harnack}) a Hölder regularity by a standard argument,
see e.g.\
\cite[Appendix~B]{guerand-imbert-2021-preprint-log-transform}.

In the general setting, \textcite{rothschild-stein-1977-hypoelliptic}
show that every hypoelliptic operator can be approximated locally by
an operator with a suitable scaling and group structure, see also
\cite{sanchez-calle-1984-fundamental,bramanti-brandolini-lanconelli-uguzzoni-2010-non-hoermander}
for use of this idea in order to obtain estimates on the smooth
problem. The application to this general setting will be explained in
a forthcoming paper.

\subsection{Comparison with literature}

As far as we are aware, there are no results in this general setting
for rough coefficients. Even in the more studied special case of
kinetic (or Kolmogorov) equations our proofs appear to be new and for
the supremum bound (\cref{thm:supremum-bound}) it appears that we
require less integrability on the coefficients as e.g. in
\cite{anceschi-rebucci-2021-preprint-weak-regularity-kolmogorov}
(other works for the supremum bound are
\cite{pascucci-polidoro-2003-gaussian-upper-bound,pascucci-polidoro-2004-mosers,cinti-polidoro-2008-pointwise-gaussian-upper-bounds,wang-zhang-2009-ultraparabolic,wang-zhang-2011-ultraparabolic,golse-imbert-mouhot-vasseur-2019-harnack-inequality,guerand-mouhot-2021-preprint-quantitative-de-giorgi,anceschi-polidoro-ragusa-2019-mosers-kolmogorov}).

For the proof of the weak Harnack inequality we use a log transform as
it already appears in the early work by
\textcite{nash-1958-continuity} on rough coefficients. This has been
used heavily for the study of equations with rough coefficients
\cite{moser-1961-harnacks,moser-1964-harnack,kruzhkov-1963,kruzhkov-1968-apriori}
and has also been used in the kinetic and Kolmogorv setting
\cite{guerand-imbert-2021-preprint-log-transform,anceschi-rebucci-2021-preprint-weak-regularity-kolmogorov}. Here
we differ by using the dual problem to conclude the result (instead of
a Poincaré inequality inspired by the framework of
\cite{hoermander-1967-hypoelliptic,albritton-armstrong-mourrat-novack-2019-preprint-variational-fokker-planck}).

\section{Preliminaries}
\label{sec:preliminaries}

For the rough operator \(\oP\), we define the principal part of
\(\oP\) as
\begin{equation*}
  \oL_p = X_i^t a^{ij} X_j.
\end{equation*}
When deriving estimates, note that \(X_i^t\) satisfies the chain rule
\begin{equation*}
  X_i^t(\alpha\, \beta) = X_i(\alpha)\, \beta + \alpha\, X_i^t(\beta).
\end{equation*}

As a first step, we note how a subsolution behaves under a
composition.
\begin{lemma}\label{thm:composition-subsolution}
  Let \(\Phi : \R \to \R\) be a smooth function with \(\Phi'\ge 0\)
  and \(\Phi'' \ge 0\).  Suppose the operator \(\oP\) of the form
  \eqref{eq:def-p} and a subsolution \(\oP u \le 0\). Then
  \(v = \Phi \circ u\) satisfies
  \begin{equation*}
    \tilde \oP v
    := X_0v
    - X_i^t\tilde{A}^i(t,x,v+h,\vec{X}v)
    - \tilde{B}(t,x,v+h,\vec{X}v)
    \le -\Phi''(u)\, \frac{\lambda}{2} |\vec{X} u|^2
    \le 0
  \end{equation*}
  where \(\tilde A\) and \(\tilde B\) are of the same form
  \eqref{eq:def-a} and \eqref{eq:def-b}, respectively, with new
  coefficients
  \begin{align*}
    \tilde a^{ij}
    &= a^{ij},&
    \tilde b^i
    &= 0,\\
    \tilde c^i &= c^i,&
    \tilde d
    &= 0,\\
    \tilde f^i
    &= \Phi'(u) \left(f^i - b^i u\right),&
    \tilde{g}
    &=
      \Phi'(u)
      \left(g - d u\right)
      - \frac{\Phi''(u)}{2\lambda}
      |f - bu|^2.
  \end{align*}
  The same result holds if \(\oP u \ge 0\) and \(\Phi'\le 0\) and
  \(\Phi'' \ge 0\).
\end{lemma}
\begin{proof}
  Using that \(\oP u \le 0\) and \(\Phi' \ge 0\), we find
  \begin{align*}
    &(\partial_t + X_0 - \oL_p)(v) \\
    &= \Phi'(u)\, (X_0 - \oL_p)(u)
      - \Phi''(u)\, a^{ij} X_iu X_ju \\
    &\le
      \Phi'(u)
      \left[
      X_i^t(b^i u - f^i) + (c^i X_i u + du-g)
      \right]
      - \Phi''(u)\, a^{ij} X_iu X_ju \\
    &= -X_i^t(\tilde f) - (b^iu-f^i) \Phi''(u) X_i u
      + c^i X_i v
      + \Phi'(u)\, (du-g)
      - \Phi''(u) a^{ij} X_i u X_j u.
  \end{align*}
  Using the square control \(\Phi''(u) a^{ij}X_iuX_ju\), we estimate
  \begin{equation*}
    - \Phi''(u)
    \left[
      \left(b^i u - f^i\right) X_iu - a^{ij}X_iu X_ju
    \right]
    \le \Phi''(u)
    \left[
      \frac{\left|f-bu\right|^2}{2\lambda}
      - \frac{\lambda}{2} |\vec{X} u|^2
    \right],
  \end{equation*}
  which then yields the result. The case \(\oP u \ge 0\) and
  \(\Phi'\le 0\) and \(\Phi'' \ge 0\) follows in the same way.
\end{proof}

In the proof of the supremum bound (\cref{thm:supremum-bound}), we
need several spatial cutoffs \(\eta\) and temporal cutoffs \(\tau\)
within the overall set \(\cyl_{S,R}(t_0,x_0)\).

For the temporal cutoff \(\tau\) between times \(s_1<s_2\),
i.e. \(\tau : \R \to [0,1]\) with \(\tau(t)=0\) for \(t\le s_1\) and
\(\tau(t)=1\) for \(t \ge s_2\), we can rescale a standard cutoff and
therefore have uniformly
\begin{equation*}
  \| \partial_t \tau \|_\infty \lesssim \frac{1}{s_2-s_1}.
\end{equation*}

For the spatial cutoff \(\eta : \R^{1+n} \to [0,1]\) between radii
\(r_1<r_2\) around \((t,x)\) over a time length \(S\), we impose that
\(X_0 \eta = 0\) and \(\eta=0\) outside \(\cyl_{S,r_2}(t,x)\) and
\(\eta=1\) inside \(\cyl_{S,r_1}(t,x)\). These can be constructed by
taking a cutoff \(\bar \eta\) between the balls \(B_{r_1}(x)\) and
\(B_{r_2}(x)\) and taking \(\eta\) as solution to
\begin{equation*}
  \left\{
    \begin{aligned}
      &X_0 \eta = 0,\\
      &\eta(t,\cdot) = \bar \eta.
    \end{aligned}
  \right.
\end{equation*}
By the definition of the parabolic cylinder, this yields a required
cutoff. Moreover, as it is always constructed within a fixed bounded
set, the smoothness of the vector fields implies
\begin{equation*}
  \| \vec X \eta \|_\infty \lesssim \frac{1}{r_2-r_1}.
\end{equation*}

\section{Local supremum bound}
\label{sec:l-inf}

In this section we prove \cref{thm:supremum-bound} by the de~Giorgi
method using the bound \hypref{h:smooth-ivp} on the smooth problem. In
the special setting of kinetic or Kolmogorov equations, the knowledge
of the fundamental solution for the smooth problem has been used in
\cite{pascucci-polidoro-2004-mosers,guerand-mouhot-2021-preprint-quantitative-de-giorgi,anceschi-rebucci-2021-preprint-weak-regularity-kolmogorov}.
In this setting the main difference is that these works use a Moser
iteration and do not obtain the integrability assumptions on the
coefficients.

The classical idea is to consider \((u-h_k)_+\) for a sequence
of cutoffs \((h_k)_k\) on nested cylinders
\(C_1 \supset C_2 \supset \dots\) and deduce that
\(\|(u-h_k)_+\| \to 0\) for a suitable norm while
\(h_k \uparrow D < \infty\). In the non-degenerate setting, this
convergence is obtained by a direct energy estimate which yields by
Sobolev embedding a gain of integrability.

In our setting, we not only perform a direct energy estimate but also
compare the subsolution of the rough problem to a solution of the
smooth problem. Hence a simple truncation is not sufficient and we
need a smoothed cutoff.

Let \(\rho \in \fsC^\infty(\R)\) be a non-negative mollification
kernel with \(\supp \rho \subset [-1,1]\) and set for \(\epsilon > 0\)
\begin{equation*}
  \rho_\epsilon(z) = \frac{1}{\epsilon} \rho\left(\frac{z}{\epsilon}\right).
\end{equation*}
As replacement for the truncation, we then define for \(h \in \R\) and
\(\epsilon>0\) the function \(K_{\epsilon,h} \in \fsC^\infty(\R)\) by
\begin{equation*}
  K_{\epsilon,h}(z) = \rho_{\epsilon} \conv (z-h)_+.
\end{equation*}

By considering \(K_{\epsilon,h}(u)\) instead of the truncation
\((u-h)_+\), we find the gain of integrability in the following lemma.

\begin{lemma}\label{thm:gain-integrability}
  Assume \hypref{h:smooth-ivp} on \(\Omega_t\) and let \(q_0,q_1\) be
  the integrabilities given by
  \begin{equation*}
    \frac 12 + \frac{1}{q_0} = \frac{1}{\gamma_0}
    \text{ and }
    \frac 12 + \frac{1}{q_1} = \frac{1}{\gamma_1}.
  \end{equation*}

  Then there exists a constant \(C_2\) with the following gain of
  integrability: for nested parabolic cylinder
  \(\cyl_{s,r}(t_0,x_0) \subset \cyl_{S,R}(t_0,x_0) \subset\subset
  \Omega_t\) and \(u\) with \(\oP u \le 0\) over
  \(\cyl_{S,R}(t_0,x_0)\), the composition
  \begin{equation*}
    v=K_{\epsilon,h}(u)
  \end{equation*}
  satisfies for any \(\epsilon>0\) and \(h \in \R\) that
  \begin{equation*}
    \begin{aligned}
      \| v \|_{\fsL^{p_1}(\cyl_{s,r})}
      &\le C_2\,
        \left[
        1 + \frac{1}{S-s} + \frac{1}{R-r}
        \right]^2
        \left[1 + \| \Lambda \|_{\fsL^{q_1}(M)}
        + \| c \|_{\fsL^{q_0}(M)}
        \right]\\
      &\qquad \left\{
        \| (1+\Lambda) v \|_{\fsL^2(M)}
        + \| \bar f \|_{\fsL^2(M)}
        + \| c v \|_{\fsL^2(M)}
        + \sqrt{\|\bar g v \|_{\fsL^1(M)}}
        \right\} \\
      &+ C_2\,
        \left[
        1 + \frac{1}{S-s} + \frac{1}{R-r}
        \right]
        \left(
        \| \bar g \|_{\fsL^{\gamma_0}(M)}
        + \frac{1}{\epsilon} \| \bar f \ind_{v\le 2\epsilon} \|_{\fsL^{2\gamma_0}(M)}^2
        \right)
    \end{aligned}
  \end{equation*}
  where \(\bar f = f-bu\) and \(\bar g = g - du\) and
  \begin{equation*}
    M = \cyl_{S,R}(t_0,x_0) \cap \{v > 0\}.
  \end{equation*}
\end{lemma}

\begin{proof}
  First note that
  \begin{align*}
    K_{\epsilon,h}'(z) &= \rho_{\epsilon} \conv \ind_{[h,\infty)} \ge
                         0, &&\text{and}&
    K_{\epsilon,h}''(z) &= \rho_{\epsilon} \conv \delta_{h}
                          = \frac{1}{\epsilon}
                          \rho\left(\frac{z-h}{\epsilon}\right)
                          \ge 0.
  \end{align*}
  Hence we can apply \cref{thm:composition-subsolution} to find that
  \(v\) satisfies with new coefficients
  \(\tilde a, \tilde b, \tilde c, \tilde d, \tilde f, \tilde g\)
  \begin{equation}
    \label{eq:sub-v}
    \tilde \oP v \le 0.
  \end{equation}

  The results is then obtained in two steps as illustrated in
  \cref{fig:sketch-integrability} with intermediate scale
  \begin{equation*}
    s < s_1 < S \qquad \text{and} \qquad r < r_1 < R
  \end{equation*}
  and corresponding cylinders
  \(\cyl_{s,r} \subset \cyl_{s_1,r_1} \subset \cyl_{S,R}\) (always
  with the base point \((t_0,x_0)\) which we therefore omit within
  this proof).  Performing a \(\fsL^2\) energy estimate with a cutoff
  from \(\cyl_{s_1,r_1}\) to \(\cyl_{S,R}\) we obtain the control
  \begin{equation}\label{eq:int:l2}
    \| \vec{X}v \|_{\fsL^2(\cyl_{s_1,r_1})}^2
    \lesssim
    \left(1 + \frac{1}{R{-}r_1} + \frac{1}{S{-}s_1}\right)^2
    \| (1+\Lambda) v \|_{\fsL^2(\cyl_{s_1,r_1})}^2
    + \| \tilde f \|_{\fsL^2(\cyl_{s_1,r_1})}^2
    + \| \tilde c v \|_{\fsL^2(\cyl_{s_1,r_1})}^2
    + \| \tilde g v \|_{\fsL^1(\cyl_{s_1,r_1})}.
  \end{equation}
  By this gained control, we can compare with the solution of the
  smooth problem and gain the claimed control in \(\cyl_{s,r}\).
  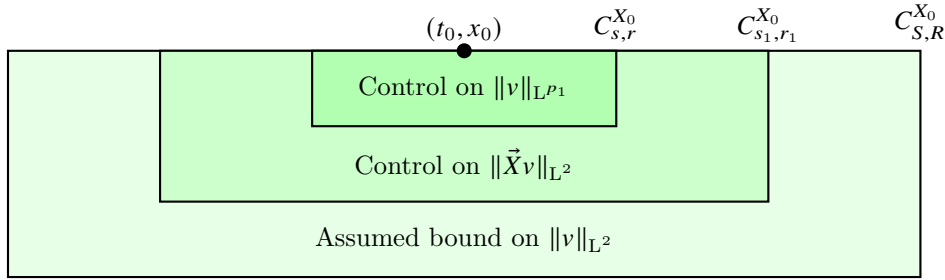
\begin{figure}[h]
    \centering
    \begin{tikzpicture}[xscale=2]
      \draw[thick,fill=green!10] (-3,0) -- (3,0) -- (3,-3) -- (-3,-3)
      -- cycle;
      \draw[thick,fill=green!20] (-2,0) -- (2,0) -- (2,-2) -- (-2,-2)
      -- cycle;
      \draw[thick,fill=green!30] (-1,0) -- (1,0) -- (1,-1) -- (-1,-1)
      -- cycle;
      \node at (0,0)[circle,fill,inner sep=2pt]{};
      \node at (0,0)[anchor=south]{\((t_0,x_0)\)};
      \node[anchor=south] at (1,0) {\(\cyl_{s,r}\)};
      \node[anchor=south] at (2,0) {\(\cyl_{s_1,r_1}\)};
      \node[anchor=south] at (3,0) {\(\cyl_{S,R}\)};
      \node at (0,-0.5) {Control on \(\| v \|_{\fsL^{p_1}}\)};
      \node at (0,-1.5) {Control on \(\| \vec{X} v \|_{\fsL^2}\)};
      \node at (0,-2.5) {Assumed bound on \(\| v \|_{\fsL^2}\)};
    \end{tikzpicture}
    \caption{Illustration of the strategy of proof for
      \cref{thm:gain-integrability}.}
    \label{fig:sketch-integrability}
  \end{figure}

  \minisec{Step 1: \(\fsL^2\) estimate}

  As discussed in \cref{sec:preliminaries} take a spatial cutoff \(\eta_1\)
  between \(r_1\) and \(R\) and a temporal cutoff \(\tau_1\) between
  \(t_0-s_1\) and \(t_0-S\).

  We now test \eqref{eq:sub-v} against \(\tau_1 \eta_1^2v\).  For the
  drift note that (using \(X_0 \eta_1 = 0\))
  \begin{equation*}
    \int_{\cyl_{S,R}}
    X_0(v)\, \tau_1 \eta_1^2 v
    = \int_{\{t_0\} \times B_R(x_0)}
    \eta_1^2 \frac{v^2}{2}
    - \int_{\cyl_{S,R}} \tau_1' \eta_1^2 \frac{v^2}{2}
    - \int_{\cyl_{S,R}} \tau_1 \eta_1^2 \frac{v^2}{2}\, \divergence X_0.
  \end{equation*}

  For the operator \(\tilde A\) note that (recalling \(\tilde b \equiv 0\))
  \begin{align*}
    p_i \tilde A^i(t,x,z,p)
    \ge \lambda |p|^2 - |p|\, |\tilde f|
    \ge \frac{3\lambda}{4} |p|^2
      - \lambda^{-1} |\tilde f|^2
  \end{align*}
  and
  \begin{equation*}
    |q_i \tilde A^i(t,x,z,p)|
    \le |q|
    \left(
      \Lambda(t,x)\, |p| + |\tilde f|
    \right).
  \end{equation*}
  Hence
  \begin{align*}
    &\int_{\cyl_{S,R}} -X_i^t \tilde A^i(t,x,v+h,\vec{X}v)\, \tau_1 \eta_1^2 v\\
    &= \int_{\cyl_{S,R}} X_i(v) \tilde A^i(t,x,v+h,\vec{X}v)\, \tau_1 \eta_1^2
      + 2 \int_{\cyl_{S,R}} X_i(\eta_1) \tilde A^i(t,x,v+h,\vec{X}v)\, \tau_1 \eta_1
      v \\
    &\ge \frac{\lambda}{2} \int_{\cyl_{S,R}} |\vec{X} v|^2 \tau_1 \eta_1^2
      - \int_{\cyl_{S,R}}
      |\tilde f|^2 \tau_1 \eta_1^2
      - \int_{\cyl_{S,R}} \left(1+4\frac{\Lambda(t,x)^2}{\lambda}\right)
      |\vec{X} \eta|^2 \tau_1 v^2.
  \end{align*}

  Finally for \(\tilde B\), we find that (recalling \(\tilde d \equiv 0\))
  \begin{equation*}
    \int_{\cyl_{S,R}} -B(t,x,v+h,\vec{X}v)\, \tau_1 \eta_1^2 v
    \ge - \int_{\cyl_{S,R}}
    \left[
      \frac{\lambda}{4}\delta |\vec{X} v|^2
      +
      |\tilde c v|^2
      + |\tilde g v|
    \right] \tau_1 \eta_1^2.
  \end{equation*}
  Hence combining the different parts yields the claimed control
  \eqref{eq:int:l2} on \(\cyl_{s_1,r_1}\).

  \minisec{Step 2: comparison with smooth problem}

  For the next step, take a spatial cutoff \(\eta_1\)
  between \(r\) and \(r_1\) and a temporal cutoff \(\tau_1\) between
  \(t_0-s\) and \(t_0-s_1\).

  The idea is to rewrite \eqref{eq:sub-v} for \(v\) as
  \begin{equation}
    \label{eq:principle-v}
    (X_0 - \oL_p)(\tau_2 \eta_2 v)
    \le \tilde G + X_i^t \tilde F^i.
  \end{equation}
  By \hypref{h:smooth-ivp}, we then find a function \(w\) solving
  \begin{equation}
    \label{eq:sup:used-dual}
    \left\{
      \begin{aligned}
        &(X_0 - \oL_0^\ext)(w)
        \ge \tilde G + X_i^t \tilde F^i + (\oL_p-\oL_0^\ext)(\tau_2 \eta_2 v)&
        &\text{in } \Omega_t^\ext \cap \{t>t_0-s_1\}, \\
        &w \ge 0
        & &\text{on } \{t_0-s_1\} \times B^\ext \cup (t_0-s_1,t_2) \times
        \partial B^\ext
      \end{aligned}
    \right.
  \end{equation}
  By the weak maximum principle for \(X_0 - \oL_0^\ext\), we find
  \begin{equation*}
    \tau_2 \eta_2 v \le w
    \qquad \text{in } \Omega_t^\ext \cap \{t>t_0-s_1\}
  \end{equation*}
  so that
  \(\| v \|_{\fsL^{p_1}(\cyl_{s,r})} \le \| w
  \|_{\fsL^{p_1}(\cyl_{s,r})}\). Then the result follows by the bound
  \eqref{eq:l0-estimate} in \hypref{h:smooth-ivp}.

  Hence we first compute (recalling \(\tilde b \equiv 0\) and
  \(\tilde d \equiv 0\))
  \begin{align*}
    (X_0 - \oL_p)(\tau_2\eta_2 v)
    &= \tau_2 \eta_2 (X_0-\oL_p)(v)
      + v\, X_0(\tau_2\eta_2)
      - \tau_2 a^{ij} X_i \eta_2 X_j v
      - X_i^t(a^{ij} v X_j \eta_2) \tau_2 \\
    &\le \tau_2 \eta_2
      \left[
      - X_i^t \tilde f^i
      + \tilde{c}^i X_i v - \tilde g
      \right]
      + v \eta_2 \tau_2'
      - \tau_2 a^{ij} X_i \eta_2 X_j v
      - X_i^t(a^{ij} v X_j \eta_2) \tau_2
  \end{align*}
  so that we verify \eqref{eq:principle-v} with
  \begin{align*}
    \tilde G &= \tau_2 X_i \eta_2 \tilde f^i
             + \tau_2 \eta_2 (\tilde c^i X_i v - \tilde g)
             + v \eta_2 \tau_2'
               - \tau_2 a^{ij} X_i \eta_2 X_j v
    &&\text{and}&
    \tilde F^i &=
                 - \tau_2 \eta_2 \tilde f^i - \tau_2 a^{ij} v
                 X_j \eta_2.
  \end{align*}
  For the additional term in \eqref{eq:sup:used-dual} note that
  \(\oL^\ext_0 = \oL_0\) in \(\Omega_t\) and
  \(\supp \tau_2\eta_2 v \subset \Omega_t\) to get
  \begin{align*}
    (\oL_p - \oL_0^\ext)(\tau_2 \eta_2 v)
    = X_i^t\left(
    (a^{ij}-\delta^{ij}) X_j(\tau_2\eta_2 v)
    \right)
  \end{align*}
  so that we find
  \begin{equation*}
    (X_0 - \oL_0^\ext)(w) \ge G + X_i^t F^i
  \end{equation*}
  where
  \begin{align*}
    G = \tilde G \text{ and }
    F^i = \tilde F^i + (a^{ij}-\delta^{ij}) X_j(\tau_2\eta_2 v)
    = - \tau_2 \eta_2 \tilde f^i
    + \tau_2 \eta_2 (a^{ij}-\delta^{ij}) X_j v
    - \tau_2 v \delta^{ij} X_j \eta_2.
  \end{align*}

  We now estimate
  \begin{equation*}
    \begin{aligned}
      \| G \|_{\fsL^{\gamma_0}(\Omega_t^\ext)}
      &\lesssim \frac{1}{r_1-r}
      \left(
        \| \tilde f \|_{\fsL^{\gamma_0}(\cyl_{s_1,r_1})}
        + \| \Lambda \vec{X} v \|_{\fsL^{\gamma_0}(\cyl_{s_1,r_1})}
      \right) \\
      &\quad + \frac{1}{s_1-s}
      \| v \|_{\fsL^{\gamma_0}(\cyl_{s_1,r_1})}
      + \| \tilde c \vec{X} v \|_{\fsL^{\gamma_0}(\cyl_{s_1,r_1})}
      + \| \tilde g \|_{\fsL^{\gamma_0}(\cyl_{s_1,r_1})}
    \end{aligned}
  \end{equation*}
  and
  \begin{equation*}
    \| F \|_{\fsL^{\gamma_1}(\Omega_t^\ext)}
    \lesssim
    \| \tilde f \|_{\fsL^{\gamma_1}(\cyl_{s_1,r_1})}
    + \| (1+\Lambda) \vec X v \|_{\fsL^{\gamma_1}(\cyl_{s_1,r_1})}
    + \frac{1}{r_1-r}
    \| v \|_{\fsL^{\gamma_1}(\cyl_{s_1,r_1})}.
  \end{equation*}

  Recalling \(\gamma_0 \le \gamma_1\le 2\) we therefore find
  \begin{equation*}
    \begin{aligned}
      \| w \|_{\fsL^{p_1}(\cyl_{s_1,r_1})}
      \lesssim
      &\left(1 + \frac{1}{r_1-r} + \frac{1}{s_1-s}\right)\\
      &\Big(
      \| \tilde f \|_{\fsL^2(\cyl_{s_1,r_1})}
      + \| \tilde g \|_{\fsL^{\gamma_0}(\cyl_{s_1,r_1})}
      + \| v \|_{\fsL^2(\cyl_{s_1,r_1})}
      + \big(1 + \| \Lambda \|_{\fsL^{q_1}(M)}
      + \| \tilde c \|_{\fsL^{q_0}(M)}\big)
      \| \vec{X} v \|_{\fsL^2(\cyl_{s_1,r_1})}
      \Big).
    \end{aligned}
  \end{equation*}
  This shows the claimed result (setting \(s_1 = (s+S)/2\) and
  \(r_1 = (r+R)/2\)) by using the expressions for \(\tilde f\) and
  \(\tilde g\) and noting that if \(K_{\epsilon,h}''(z)\) is zero
  unless \(|z-h|\le \epsilon\) so that
  \begin{equation*}
    \| K_{\epsilon,h}'' (f-bu)^2 v\|_{\fsL^1(M)}
    \lesssim \| \bar f \ind_{v\le 2\epsilon} \|_{\fsL^2(M)}^2
    \text{ and }
    \| K_{\epsilon,h}'' (f-bu)^2 \|_{\fsL^{\gamma_0}(M)}
    \lesssim \frac{1}{\epsilon}
    \| \bar f \ind_{v\le 2\epsilon} \|_{\fsL^{2\gamma_0}(M)}^2.
  \end{equation*}
  The restriction of \(\bar f\) and \(\bar g\) to \(M = \{v>0\}\)
  follows from the fact that the factor \(K_{\epsilon,h}'(u)\) and
  \(K_{\epsilon,h}''(u)\) in \(\tilde f\) and \(\tilde g\) vanish
  otherwise.
\end{proof}

By interpolation we can start from the \(\fsL^1\) norm.

\begin{lemma}\label{thm:l1-interpolation}
  Assume the setup of \cref{thm:gain-integrability} with \(h>0\) and
  take as in the statement of \cref{thm:supremum-bound} exponents
  \(p_0\) and \(q_d\) and let \(\bar q_0\) be the exponent given by
  \begin{equation*}
    \frac{1}{\bar q_0} + \frac{1}{p_0} = \frac 12.
  \end{equation*}
  Then there exists a constant \(C_3\) and exponent \(\alpha\) such
  that for \(0<s<S\) and \(0<r<R\) it holds that
  \begin{equation*}
    \| v \|_{\fsL^{p_1}(\cyl_{s,r})}
    \le C_3 P^\alpha \| v \|_{\fsL^{1}(\cyl_{S,R})}
    + C_3 Q
  \end{equation*}
  where (with \(p_0^*\) as the dual exponent of \(p_0\))
  \begin{equation*}
    \begin{split}
      P = &\left(1 + \frac{1}{S-s} + \frac{1}{R-r}\right)^2
            \left[1 + \| \Lambda \|_{\fsL^{q_1}(M)}
            + \| c \|_{\fsL^{q_0}(M)}
            \right]\\
          &\left[
            1 + \| \Lambda \|_{\fsL^{\bar q_0}(M)}
            + \| c \|_{\fsL^{\bar q_0}(M)}
            + \| b \|_{\fsL^{\bar q_0}(M)}
            + \| d \|_{\fsL^{q_d}(M)}
            + \| g \|_{\fsL^{p_0^*}(M)}
            \right]
    \end{split}
  \end{equation*}
  and
  \begin{equation*}
    \begin{split}
      Q
      &= \left(1 + \frac{1}{S-s} + \frac{1}{R-r}\right)^2
        \left[1 + \| \Lambda \|_{\fsL^{q_1}(M)}
        + \| c \|_{\fsL^{q_0}(M)}
        \right]\\
      &\left[
        \| f \|_{\fsL^2(M)}
        + \frac{1}{\epsilon} \| f \|_{\fsL^{2\gamma_0}(M)}^2
        + (\epsilon+h)
        \left(
        \| b \|_{\fsL^{2}(M)}
        + \frac{\epsilon+h}{\epsilon}
        \| b \|_{\fsL^{2\gamma_0}(M)}^2
        + \| d \|_{\fsL^{\gamma_0}(M)}
        \right)
        + \| g \|_{\fsL^{\gamma_0}(M)}
        \right].
    \end{split}
  \end{equation*}
\end{lemma}
\begin{proof}
  For \(\sigma \in [0,1]\) define the cylinder
  \(\cyl_\sigma = \cyl_{s_\sigma,r_\sigma}\) with
  \begin{equation*}
    s_{\sigma} = s + \sigma (S-s)
    \quad\text{and}\quad
    r_{\sigma} = r + \sigma (R-r)
  \end{equation*}
  and let
  \begin{equation*}
    Z(\sigma) = \| v \|_{\fsL^{p_1}(\cyl_\sigma)}.
  \end{equation*}

  For \(0\le\sigma_1<\sigma_2\le 1\) apply
  \cref{thm:gain-integrability} to find with a constant \(\tilde C_2\)
  \begin{equation*}
    Z(\sigma_1) \le \tilde C_2
    \left[1 + \frac{1}{\sigma_2-\sigma_1}\right]^2
    \left[P \| v \|_{\fsL^{p_0}(\cyl_{\sigma_2})} + Q\right].
  \end{equation*}

  As \(1<p_0<p_1\) there exist interpolation parameter
  \(\theta \in (0,1)\) such that for all \(\delta>0\)
  \begin{equation*}
    \| v \|_{\fsL^{p_0}} \le \| v \|_{\fsL^1}^{1-\theta} \| v \|_{\fsL^{p_1}}^{\theta}
    \le \theta \delta \| v \|_{\fsL^{p_1}}
      + (1-\theta) \delta^{\frac{-\theta}{1-\theta}}
      \| v \|_{\fsL^1}.
  \end{equation*}

  By the interpolation we find with a constant \(\tilde C_3\)
  \begin{equation*}
    Z(\sigma_1) \le \theta Z(\sigma_2) + \tilde C_3
    \left[
      1 + \frac{1}{\sigma_2-\sigma_1}
    \right]^{\frac{2}{1-\theta}}
    \left[
    P^{\frac{1}{1-\theta}}
    \| v \|_{\fsL^1(M)}
    +Q
    \right].
  \end{equation*}

  The result now follows by a standard argument for geometric series,
  see e.g.\ \cite[Lemma~6.1]{giusti-2003-direct}. Consider for some
  \(\beta<1\) the sequence
  \begin{equation*}
    \sigma_i = 1 - \beta^i.
  \end{equation*}
  Then the previous argument shows
  \begin{equation*}
    Z(\sigma_{i-1}) \le \theta Z(\sigma_i)
    + \tilde C_3
    \left[
      P^{\frac{1}{1-\theta}}
      \| v \|_{\fsL^1(M)}
      +Q
    \right]
    \left(1 + \frac{1}{(1-\beta)\beta^{i-1}}\right)^{\frac{2}{1-\theta}}
  \end{equation*}
  and iterating the argument shows that for any \(k\in\N\)
  \begin{equation*}
    Z(0) = Z(\sigma_0)
    \le \theta^{k} Z(\sigma_k)
    + \tilde C_3
    \left[
      P^{\frac{1}{1-\theta}}
      \| v \|_{\fsL^1(M)}
      +Q
    \right]
    \sum_{i=1}^{k} \theta^{i-1}
    \left(1 + \frac{1}{(1-\beta)\beta^{i-1}}\right)^{\frac{2}{1-\theta}}.
  \end{equation*}
  For \(\beta\) sufficiently close to \(1\), the series converges and
  the result follows.
\end{proof}

We can now collect the different parts.

\begin{proof}[Proof of \cref{thm:supremum-bound}]
  By considering \(u/N\) instead of \(u\) where
  \(N = \| u \|_{\fsL^{1}(\cyl_{S,R}(t_0,x_0))} + \| f
  \|_{\fsL^{q_b}(\cyl_{S,R}(t_0,x_0))} + \| g
  \|_{\fsL^{q_d}(\cyl_{S,R}(t_0,x_0))} \) it suffices to prove
  \begin{equation*}
    \sup_{\cyl_{s,r}(t_0,x_0)} u \le C_S\, (1+\delta_S)^\beta
  \end{equation*}
  under the assumption that
  \begin{equation*}
    \| u \|_{\fsL^{1}(\cyl_{S,R}(t_0,x_0))}
    + \| f \|_{\fsL^{q_b}(\cyl_{S,R}(t_0,x_0))}
    + \| g \|_{\fsL^{q_d}(\cyl_{S,R}(t_0,x_0))}
    \le 1.
  \end{equation*}\medskip

  For the proof, consider a sequence of cylinders
  \(\cyl_k = \cyl_{s_k,r_k}(t_0,x_0)\) for \(k \in \N\) where
  \begin{equation*}
    s_k = s + 2^{-k} (S-s)
    \text{ and }
    r_k = r + 2^{-k} (R-r).
  \end{equation*}
  On the cylinders consider the regularised cutoffs
  \(v_k = K_{\epsilon_k,h_k}(u)\) where
  \begin{equation*}
    h_k = D (1-2^{-k})
    \text{ and }
    \epsilon_k = \frac{D}{4} 2^{-k}
  \end{equation*}
  for a parameter \(D \ge 1\). We then study
  \begin{equation*}
    Z_k = \| v_k \|_{\fsL^{p_0}(\cyl_{k+1})}
    \text{ and }
    M_k = C_k \cap \{v_k > 0\}.
  \end{equation*}

  As a first step we will then show for an exponent \(\alpha_1>0\)
  the initial bound
  \begin{equation}
    \label{eq:sup:iteration-start}
    Z_1 \lesssim (1+\delta_S)^{\alpha_1}\,
    D^{1-\frac{p_1}{p_0}}.
  \end{equation}
  The second step is to show for a constant \(W\) and exponents
  \(\alpha_2,\delta>0\) that
  \begin{equation}
    \label{eq:sup:iteration-step}
    Z_{k} \lesssim W^{k} (1+\delta_S)^2 Z_{k-1}^{1+\delta}.
  \end{equation}
  Hence for \(D \gtrsim (1+\delta_S)^\beta\) for some exponent
  \(\beta>0\) we have that \(Z_k \to 0\) as \(k\to \infty\) which
  implies the result.

  \minisec{Initial bound \eqref{eq:sup:iteration-start}}

  This follows from applying \cref{thm:l1-interpolation} twice. As a
  first step apply it between the cylinders \(\cyl_1\) and \(\cyl_0\)
  with \(K_{0,1}(u)\) to find
  \begin{equation*}
    \| K_{0,1}(u) \|_{\fsL^{p_1}(\cyl_1)} \lesssim (1+\delta_S)^{2\alpha}.
  \end{equation*}
  By Hölder this implies for \(D\) large enough that
  \begin{equation*}
    |M_1| \lesssim D^{-\frac{1}{p_1}} (1+\delta_S)^{2\alpha}.
  \end{equation*}
  Using again Hölder, this shows by the choice of the integrabilities
  that for \(D\) large enough
  \begin{equation*}
    (\epsilon_1+h_1)
    \left(
      \| b \|_{\fsL^{2}(M_1)}
      + \frac{\epsilon+h}{\epsilon}
      \| b \|_{\fsL^{2\gamma_0}(M_1)}^2
      + \| d \|_{\fsL^{\gamma_0}(M_1)}
    \right)
    \lesssim (1+\delta_S)^2.
  \end{equation*}
  Hence we can apply \cref{thm:l1-interpolation} again to find with
  some exponent \(\alpha_1\)
  \begin{equation*}
    \| v_1 \|_{\fsL^{p_1}(\cyl_1)} \lesssim
    (1+\delta_S)^{\alpha_1}.
  \end{equation*}
  This yields \eqref{eq:sup:iteration-start} by another application of Hölder.

  \minisec{Iteration step \eqref{eq:sup:iteration-step}}

  Note that the regularisations \(\epsilon_k\) are chosen such that
  \(\epsilon_k + \epsilon_{k+1} \le (h_{k+1}-h_k)/2\) so that
  \begin{equation}
    \label{eq:sup:comparision-steps}
    \{v_{k+1}>0\} \subset \{v_{k} > (h_{k+1}-h_k)/2\}
    \text{ and }
    v_{k+1} \ge v_k.
  \end{equation}
  Hence we find by Hölder that
  \begin{equation}
    \label{eq:sup:m-z}
    |M_k|^{1/p_0}
    \le \frac{2}{h_k-h_{k-1}} Z_{k-1}.
  \end{equation}
  and
  \begin{equation}
    \label{eq:sup:z-1}
    Z_k \le |M_k|^{\frac{1}{p_0}-\frac{1}{p_1}} \| v_k \|_{\fsL^{p_1}(\cyl_{k+1})}.
  \end{equation}

  By \cref{thm:gain-integrability} we estimate
  \(\| v_k \|_{\fsL^{p_1}(\cyl_{k+1})}\) by going to \(\cyl_k\) as
  \begin{equation}
    \label{eq:sup:z-2}
    \| v_k \|_{\fsL^{p_1}(\cyl_{k+1})}
    \lesssim 8^k (1+\delta_S)^2
    \left[
      (1+D) |M_k|^{\frac{1}{p_0}} + \| v_k \|_{\fsL^{p_0}(\cyl_k)}
    \right]
  \end{equation}
  by noting the bounds
  \begin{align*}
    \| (1+\Lambda) v_k \|_{\fsL^2(M_k)}
    &\le (1 + \| \Lambda \|_{\fsL^{\bar q_0}})
      \| v_k \|_{\fsL^{p_0}(\cyl_k)}\\
    \| \bar f_k \|_{\fsL^2(M_k)}
    &\le \| f \|_{\fsL^{\bar q_0}}
      |M_k|^{\frac 12 - \frac{1}{\bar q_0}}
      + \| b \|_{\fsL^{\bar q_0}}
      \left( (1+D) |M_k|^{\frac 12 - \frac{1}{\bar q_0}}
      + \| v_k \|_{\fsL^{p_0}(M_k)} \right) \\
    \| c v_k \|_{\fsL^2(M_k)}
    &\le \| c \|_{\fsL^{\bar q_0}}
      |M_k|^{\frac 12 - \frac{1}{\bar q_0}}\\
    \| \bar g_k v_k \|_{\fsL^1(M_k)}^{1/2}
    &\lesssim \| \bar g_k \|_{\fsL^{p_0^*}(M_k)}
      + \| v_k \|_{\fsL^{p_0}(M_k)}\\
    &\lesssim \left( \|g \|_{\fsL^{q_d}} + (1+D) \| d \|_{\fsL^{q_d}}
      \right) |M_k|^{\frac{1}{p_0^*}-\frac{1}{q_d}}
      + (1 + \| d \|_{\fsL^{q_d}})
      \| v_k \|_{\fsL^{p_0}(M_k)}\\
    \| \bar g_k \|_{\fsL^{\gamma_0}(M_k)}
    &\le \| g \|_{\fsL^{q_d}}
      |M_k|^{\frac{1}{\gamma_0} - \frac{1}{q_d}}
      + \| d \|_{\fsL^{q_d}}
      \left( (1+D) |M_k|^{\frac{1}{\gamma_0} - \frac{1}{q_d}}
      + \| v_k \|_{\fsL^{p_0}(M_k)} \right) \\
    \frac{1}{\epsilon_k} \| \bar f_k \ind_{v_k\le 2\epsilon_k}
    \|_{\fsL^{2\gamma_0}(M_k)}^2
    &\lesssim 2^k (1+D)
      \left(
      \| f \|_{\fsL^{q_b}}^2 + \| b \|_{\fsL^{q_b}}^2
      \right)
      |M_k|^{2\left(\frac{1}{2\gamma_0}-\frac{1}{q_b}\right)}.
  \end{align*}

  Chaining \eqref{eq:sup:m-z}, \eqref{eq:sup:z-1} and
  \eqref{eq:sup:z-2} then yields the required bound
  \eqref{eq:sup:iteration-step}.
\end{proof}

\section{Weak Harnack inequality}
\label{sec:harnack}

In this section we prove \cref{thm:weak-harnack}. We first introduce a
regularised version of \((-\log(z))_+\) as \(G \in
\fsC^2((0,\infty))\) by
\begin{equation*}
  G(z) = (-\log z +z - 1) \ind_{z\le 1}
\end{equation*}
so that
\begin{equation*}
  G'(z) = \left(-\frac 1z + 1\right) \ind_{z\le 1} \le 0
  \text{ and }
  G''(z) = \frac{1}{z^2} \ind_{z\le 1} \ge 0.
\end{equation*}
This implies
\begin{equation}
  \label{eq:g-magic}
  G''(z) \ge \left[G'(z)\right]^2.
\end{equation}

As already used in \textcite{nash-1958-continuity}, then consider
\begin{equation*}
  v = G_\delta(u)
  \quad\text{where}\quad
  G_{\delta}(z) = G\left(\frac{z+\delta}{1+\delta}\right)
\end{equation*}
for a small enough \(\delta > 0\). As \(u \ge 0\) we have the trivial
bound
\begin{equation}
  \label{eq:weak-harnack:trivial}
  |v| \le G_{\delta}(0).
\end{equation}

The strategy is to use \eqref{eq:g-magic} and
\eqref{eq:weak-harnack:trivial} in order to gain a control of the
rough form \(\| \vec{X} v\|_{\fsL^2}^2 \lesssim G_{\delta}(0)\). In
the parabolic case, we can reinterpret the argument by
\textcite{nash-1958-continuity} as using the information that
\(u \ge 1\) in \(E\) to conclude by a variation of Poincaré and the
supremum bound that
\(\| v \|_{\fsL^\infty} \lesssim \| \vec{X} v\|_{\fsL^2} \lesssim
\sqrt{G_{\delta}(0)}\). As we have gained a square root, we can then
make \(\delta\) sufficiently small to conclude a non-trivial bound on
\(\| v \|_{\fsL^\infty}\) which yields the statement. In the classical
non-degenerate setting, this ideas has been used in
\cite{moser-1961-harnacks,moser-1964-harnack,kruzhkov-1963,kruzhkov-1968-apriori}
and in the kinetic and Kolmogorv setting
\cite{wang-zhang-2009-ultraparabolic,wang-zhang-2011-ultraparabolic,guerand-imbert-2021-preprint-log-transform,anceschi-rebucci-2021-preprint-weak-regularity-kolmogorov}. Here
we differ by using the dual problem to conclude the result (instead of
a Poincaré inequality inspired by the framework of
\cite{hoermander-1967-hypoelliptic,albritton-armstrong-mourrat-novack-2019-preprint-variational-fokker-planck}).

As \(u\) is a supersolution, we apply
\cref{thm:composition-subsolution} to conclude together with
\eqref{eq:g-magic} that
\begin{equation}
  \label{eq:weak-harnack:evo}
  \tilde P v + \frac{\lambda}{2} |\vec X v|^2 \le 0
\end{equation}
where \(\tilde P\) is the operator in \eqref{eq:def-p} with the new
coefficients \(\tilde a, \tilde b, \tilde c, \tilde d, \tilde f, \tilde g\).

\begin{figure}[h]
  \centering
  \begin{tikzpicture}[xscale=1.5,yscale=4]
    \draw[thick,fill=green!10] (-5,0) -- (5,0) -- (5,-1) -- (-5,-1)
    -- cycle;
    \node[anchor=south] at (5,0) {\(\Sigma_R\)};
    \draw[thick,fill=green!20] (-4,0) -- (4,0) -- (4,-1) -- (-4,-1)
    -- cycle;
    \node[anchor=south] at (4,0) {\(\tilde \Sigma_R\)};
    \draw[thick,fill=green!30] (-2,0) -- (2,0) -- (2,-0.5) -- (-2,-0.5)
    -- cycle;
    \node[anchor=south] at (2,0) {\(\cyl_{\frac 12,2}\)};
    \draw[thick,fill=green!30] (-1,0) -- (1,0) -- (1,-0.33) -- (-1,-0.33)
    -- cycle;
    \node[anchor=south] at (1,0) {\(\cyl_{\frac 13,1}\)};
    \node at (0,0)[circle,fill,inner sep=2pt]{};
    \node at (0,0)[anchor=south]{\((t_0,x_0)\)};
    \node at (4,-0.5)[anchor=west,align=left] {\(\|\vec X v\|^2\)\\\(\lesssim G_\delta(0)\)};
    \node at (0,-0.75) {Comparison with smooth dual problem};
    \node at (0,-0.5)[anchor=south] {\(\| v \|_{\fsL^1}\) control};
    \node at (0,-0.33)[anchor=south] {\(\| v \|_{\fsL^\infty}\) control};
  \end{tikzpicture}
  \caption{Illustration of the overview of proof for \cref{thm:weak-harnack}.}
  \label{fig:sketch-weak-harnack}
\end{figure}
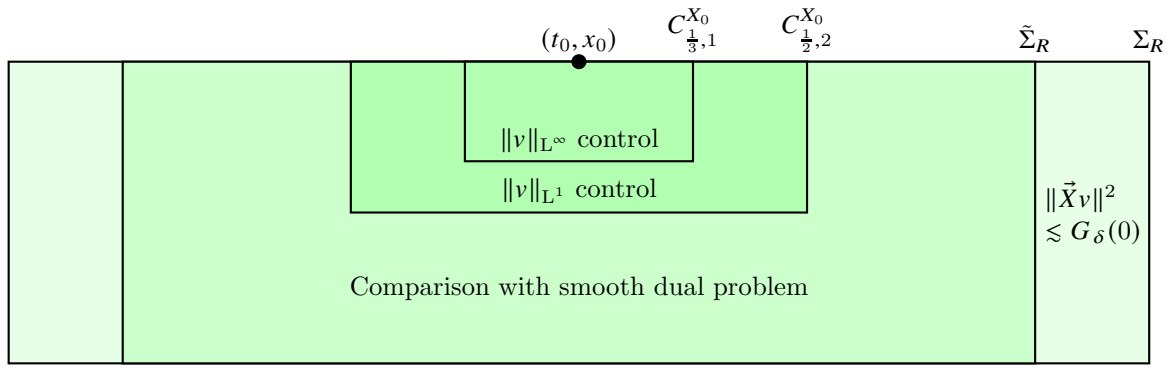

For a parameter \(R\) which is chosen large enough later, we take by
\hypref{h:small-cutoff} the sets \(\tilde \Sigma_R\) and
\(\Sigma_R\). Then we obtain the result in the following steps, cf.\
\cref{fig:sketch-weak-harnack}. By integrating the trivial \(\fsL^1\)
bound we find that
\begin{equation}
  \label{eq:weak-harnack:l1-gain}
  \| \vec X v \|_{\fsL^2(\tilde \Sigma_R)}^2 \le \tilde C_1
  \left[
    C(R) G_\delta(0) + \| \Lambda \|_{\fsL^2(\Sigma_R)}^2
    + \| \tilde f \|_{\fsL^1(\Sigma_R)}
    + \| \tilde c \|_{\fsL^2(\Sigma_R)}^2
    + \| \tilde g \|_{\fsL^1(\Sigma_R)}
  \right]
\end{equation}
for a constant \(\tilde C_1\) and a constant \(C(R)\) depending on
\(R\). Then integrating \(v\) with the solution of the dual problem
\(w\) solving \eqref{eq:weak-harnack-dual} from \hypref{h:smooth-ivp},
we find for \(\bar q_2\) from the statement of \cref{thm:weak-harnack}
and a constant \(\tilde C_2\) and a new constant \(C(R)\) depending on
\(R\) that
\begin{equation}
  \label{eq:weak-harnack:dual-gain}
  \| v \|_{\fsL^1(\cyl_{1/2,2})}^2 \le \tilde C_2
  \left(\frac{G_{\delta}(0)}{R}\right)^2
  + C(R)
  \left[
    (1 + \| \Lambda \|_{\fsL^{\bar q_2}(\tilde \Sigma_R)}
    + \| \tilde c \|_{\fsL^{\bar q_2}(\tilde \Sigma_R)} )^2
    \| \vec X v \|_{\fsL^2(\tilde \Sigma_R)}^2
    + \| \tilde f \|_{\fsL^{2}(\tilde \Sigma_R)}^2
    + \| \tilde g \|_{\fsL^{2}(\tilde \Sigma_R)}^2
  \right]
\end{equation}
and by the supremum bound we finally conclude that by choosing
\(R = C_R (1+\delta_S)^\beta\) for \(C_R\) sufficiently large that
\begin{equation}
  \label{eq:weak-harnack:sup-control}
  \begin{split}
    &\| v \|_{\fsL^\infty(\cyl_{1/3,1})}^2 -
    \left(\frac{G_\delta(0)}{2}\right)^2\\
    &\lesssim (1+\delta_S)^{2\beta} (1+\Delta)^2
      \left[
      G_\delta(0)
      + 1 + \Delta^2 +
      \frac{\|f-ub\|_{\fsL^{q_b}(\Sigma_R)}^2}{\delta^2}
      +
      \frac{\|g-ud\|_{\fsL^{q_d}(\Sigma_R)}^2}{\delta^2}
      +
      \frac{\|f-ub\|_{\fsL^{2q_d}(\Sigma_R)}^4}{\delta^4}
      \right].
  \end{split}
\end{equation}
By taking \(\delta\) small enough, we will then conclude the result.

\minisec{Integrating $\fsL^1$ bound}

Take the cutoff \(\eta_R\) from \hypref{h:small-cutoff} and consider
the localised \(\fsL^1\) norm
\begin{equation*}
  E(t) = \int_{\Sigma_t \cap (\{t\} \times \R^n)} v(t,x)\, \eta_1^2(t,x)\, \dd x.
\end{equation*}
By \eqref{eq:weak-harnack:trivial} we find as \(\Sigma_R\) is bounded
the trivial bound
\begin{equation*}
  E(t) \lesssim C(R) G_\delta(0).
\end{equation*}
By \eqref{eq:weak-harnack:evo} we find that
\begin{align*}
  \frac{\dd}{\dd t} E
  \le - \frac{\lambda}{4} \int |\vec X v|^2 \eta_R^2
  + \| \divergence X_0 \|_{\infty} E
  + \frac{2}{\lambda} \int \Lambda^2 |\vec X \eta_R|^2
  + \int |\tilde f| \, |\vec{X} \eta_R^2|
  + \frac{2}{\lambda} \int |\tilde c|^2\, \eta_R^2
  + \int |\tilde g|.
\end{align*}
Integrating over time hence yields the claimed
control~\eqref{eq:weak-harnack:l1-gain}.

\minisec{Comparison with dual problem}

Let \(w\) be the solution of the smooth dual
problem~\eqref{eq:weak-harnack-dual} where \(E\) is given in
\eqref{eq:weak-harnack-e}.

Then consider with the cutoff \(\tilde \eta_R\) from \hypref{h:small-cutoff}
\begin{equation*}
  K(t) = \int_{\Sigma_t \cap (\{t\} \times \R^n)} v\,w\, \tilde \eta_R\, \dd x.
\end{equation*}
We then find
\begin{align*}
  \frac{\dd}{\dd t} K
  &\le
    - \int X_jv a^{ij} X_i(w \tilde \eta_R)
    - \int X_jv \delta^{ij} \left[ \tilde \eta_R X_i w - w X_i \tilde \eta R\right]
    + \int v w \delta^{ij} X_i X_j \tilde \eta_R\\
  &\quad+ \int \tilde f^i X_i(w \tilde \eta_R)
    + \int \tilde c^i X_i v\, w \tilde \eta_R
    + \int \tilde g w \tilde \eta_R
    + \int v \ind_{E}.
\end{align*}
By construction \(v(x)=0\) if \(x \in E\) so that the last term
vanishes and \(w\equiv 0\) for \(t=-1\). Hence integrating
\(t\in [0,1]\) yields (using \hypref{h:small-cutoff} for bounding
\(\delta^{ij} X_i X_j \tilde \eta_R\))
\begin{align*}
  K(t)
  &\lesssim \|\vec X v\|_{\fsL^2(\tilde \Sigma_R)}
    \left(
    \| (1+\Lambda) \vec{X} w \|_{\fsL^2(\tilde \Sigma_R)}
    + \| (1+\Lambda) w \|_{\fsL^2(\tilde \Sigma_R)}
    + \| \tilde c w \|_{\fsL^2(\tilde \Sigma_R)}
    \right)
    + G_\delta(0)\, \| w \|_{\fsL^1([-1,0]\times \R^n)} R^{-1}\\
  &+ \| \tilde f \|_{\fsL^{2}(\Sigma_R)}
    \left( \| w \|_{\fsL^{p_2}(\Sigma_R)}
    + \| Xw \|_{\fsL^{p_2}(\Sigma_R)} \right)
    + \| \tilde g \|_{\fsL^{2}(\Sigma_R)}
    \| w \|_{\fsL^{p_2}(\Sigma_R)}
\end{align*}
By \hypref{h:dual-spreading} we find that
\(\| v \|_{\fsL^1(\cyl_{1/2,2})} \le \mu_0^{-1} \sup_{t\in[-1/2,0]}
K(t)\) so that the claimed estimate~\eqref{eq:weak-harnack:dual-gain}
follows by the bounds of \hypref{h:dual-spreading}.

\minisec{Using the supremum bound and conclusion}

By using the supremum bound (\cref{thm:supremum-bound}) between
\(\cyl_{\frac 13,1}\) and \(\cyl_{\frac 12,2}\), we find with a
constant \(\tilde C_3\) and a new constant \(C(R)\) depending on \(R\) that
\begin{equation*}
  \| v \|_{\fsL^\infty(\cyl_{\frac 13,1})}^2
  - \tilde C_3 (1+\delta_S)^{2\beta} \left(\frac{G_\delta(0)}{R}\right)^2
  \le C(R)  (1+\delta_S)^{2\beta}
  \left[ (1+\Delta)^2 \| \vec X v \|_2^2
    + \| \tilde f \|_{\fsL^{q_b}(\Sigma_R)}^2
    + \| \tilde g \|_{\fsL^{q_d}(\Sigma_R)}^2
  \right].
\end{equation*}
By choosing \(C_R\) sufficiently large and setting \(R = C_R
(1+\delta_S)^{\beta}\), we find
\begin{equation*}
  \| v \|_{\fsL^\infty(\cyl_{\frac 13,1})}^2
  - \left(\frac{G_\delta(0)}{2}\right)^2
  \lesssim  (1+\delta_S)^{2\beta}
    \left[ (1+\Delta)^2 \| \vec X v \|_2^2
    + \| \tilde f \|_{\fsL^{q_b}(\Sigma_R)}^2
    + \| \tilde g \|_{\fsL^{q_d}(\Sigma_R)}^2
    \right].
\end{equation*}
Plugging in~\eqref{eq:weak-harnack:l1-gain} then gives the claimed
estimate~\eqref{eq:weak-harnack:sup-control}.

As \(G_\delta(0) \to \infty\) as \(\delta \to 0\), we can then find a
sufficiently small \(\delta\) such
that~\eqref{eq:weak-harnack:sup-control} becomes
\begin{equation*}
  \| v \|_{\fsL^\infty(\cyl_{\frac 13,1})} - \frac{3}{4} G_\delta(0)
  \le \epsilon.
\end{equation*}
By letting \(\epsilon\) small enough this shows that
\(v \le (5/4) G_{\delta}(0)\) in \(\cyl_{\frac 13,1}\). As the relation
\(G_\delta(u) = v \le (5/4) G_{\delta}(0)\) implies \(u \ge \mu\)
for a constant \(\mu\), this shows the result.

\section*{Acknowledgements}

The first author would like to thank the mathematical department of
the University of Leipzig for the possibility of a long visit during
which this work has started and the continued hospitality. He would
like to thank the Isaac Newton Institute for Mathematical Sciences,
Cambridge, for support and hospitality during the programme
``Frontiers in Kinetic Theory''.  This was supported by EPSRC grant no
EP/R014604/1 and a grant from the Simons Foundations.  He also
acknowledges the grant ANR-18-CE40-0027 of the French National
Research Agency (ANR).
The second author was partially supported by the German Science Foundation DFG in context of the Priority Program SPP 2026 ``Geometry at Infinity''.

\appendix

\section{Notes on function spaces}
\label{sec:function-spaces}

We first note that \(u \in \fsHyp^1\) and \(X_0 u \in \fsHyp^{-1}\)
implies with \hypref{h:smooth-ivp} more regularity.

\begin{lemma}
  For a domain \(\Omega_t\) suppose \hypref{h:smooth-ivp} and suppose
  $u\in \fsHyp^1(\Omega_t)$ with
  $X_0 u \in \fsHyp^{-1}(\Omega_t)$. Then
  $u \in L^{p_1}_{loc}(\Omega_t)$.
\end{lemma}
\begin{proof}
  Take a compactly supported subset \(\Omega_t'\) of \(\Omega_t\) and
  let \(\varphi\) be a smooth cutoff. Then \(\varphi u \in \fsHyp^1\)
  and \(X_0(\varphi u) \in \fsHyp^{-1}\) as we can note
  \begin{equation*}
    \| X_0(\varphi u)\|_{\fsHyp^{-1}}
    \le (\|\varphi\|_\infty + \|X_0\varphi\|_\infty +
    \|\vec{X}\varphi\|_\infty)
    (\|u\|_{\fsL^2} + \| X_0u\|_{\fsHyp^{-1}}).
  \end{equation*}

  Due to embedding
  $u \in \fsHyp^1 \mapsto (u, \vec{X}u) \in (\fsL^2)^{m+1}$ every
  element $\lambda \in \fsHyp^{-1}$ can be represented as
  $\lambda = - X_i^t f^i +f^0$. Hence \hypref{h:smooth-ivp} yields
  the result.
\end{proof}

After the above described control of \(u \in \fsL^{p_1}_{loc}\) all
the a priori estimates can be defined by standard methods.

Furthermore we shortly want to recall a simple argument for a weak
maximum principle in the setting of hypoelliptic operators:
\begin{lemma}\label{lem.weakmaxprinciple}
  Let $w\in \fsHyp^1(\Omega_t^\ext \cap \{t>t_\init\})$ be a weak
  subsolution of
  \begin{equation}\label{eq.weaksubsolution}
      \left\{
    \begin{aligned}
      &(X_0 - \oL_0^\ext) w \le 0
      & &\text{in } \Omega_t^\ext \cap \{t>t_\init\}, \\
      &w = 0
      & &\text{on } \{t_\init\} \times B^\ext \cup (t_\init,t_2) \times
      \partial B^\ext
    \end{aligned}
  \right.
  \end{equation}
  then $w\le 0$ a.e. in $\fsL^2(\Omega_t^\ext \cap \{t>t_\init\})$.
\end{lemma}

\begin{proof}
  Let $\epsilon >0$ and consider the non-decreasing convex function
  $K(z)=K_{\epsilon,
    2\epsilon}(z)=\rho_\epsilon*(z-2\epsilon)_+$. Since by assumptions
  $X_i w \in \fsL^2(\Omega_t^\ext \cap \{t>t_\init\})$ for
  $ i=1,\dotsc, m$, it is not hard to check that $w_\epsilon = K(w)$
  is still a weak subsolution of \eqref{eq.weaksubsolution} with
  $X_iw_\epsilon \in \fsL^2(\Omega_t^\ext \cap \{t>t_\init\})$ for
  $i=1,\dotsc, m$.

  Testing the equation with $w_\epsilon$ and using a classical
  Gronwall argument one obtains
  \[ \sup_{t>t_\init} \| w_\epsilon \|_{\fsL^2(\Omega_t^\ext \cap
      \{t=t_\init\})} \le 0\,.\] But this clearly implies
  $w_\epsilon = 0$ a.e. Since $\epsilon>0$ was chosen arbitrary the
  conclusion follows.
\end{proof}

\section{Construction of comparision function}
\label{sec:construction-comparison}

In this section, we will discuss how the Hörmander estimates can be
used to verify \hypref{h:smooth-ivp}. Here we take bounded balls \(B\)
and \(B^\ext\) and the corresponding parabolic domains \(\Omega_t\)
and \(\Omega_t^\ext\). Between the balls \(B\) and \(B^\ext\), find a
smooth cutoff \(\chi\) and consider
\begin{equation*}
  \oL_0^\ext = \oL_0 + \nabla \cdot ((1-\chi)^2 \nabla\cdot).
\end{equation*}

For any \(\varepsilon > 0\), we can then find by standard parabolic
theory or the method of continuity a solution \(w^\varepsilon\) of
\begin{equation*}
  \left\{
    \begin{aligned}
      &(X_0 - \oL_0^\ext - \epsilon \Delta) w^\varepsilon = G + \sum_{i=1}^{m} X_i^t F^i
      & &\text{in } \Omega_t^\ext \cap \{t>t_\init\}, \\
      &w^\varepsilon = 0
      & &\text{on } \{t_\init\} \times B^\ext \cup (t_\init,t_2) \times
      \partial B^\ext
    \end{aligned}
  \right.
\end{equation*}

with the uniform \(\fsL^2\) estimate
\begin{equation*}
  \begin{split}
    &\| w^\varepsilon \|_{\fsL^2(\Omega_t^\ext \cap \{t>t_\init\})}
    + \| \vec X w^\varepsilon \|_{\fsL^2(\Omega_t^\ext \cap \{t>t_\init\})}
    + \| (1-\chi) \nabla w^\varepsilon \|_{\fsL^2(\Omega_t^\ext \cap \{t>t_\init\})} \\
    &\quad \lesssim
    \| F \|_{\fsL^2(\Omega_t^\ext \cap \{t>t_\init\})}
    +
    \| G \|_{\fsL^2(\Omega_t^\ext \cap \{t>t_\init\})}.
  \end{split}
\end{equation*}
By compactness, we can therefore find a weak limit
\(w \in \fsL^2(\Omega_t^\ext \cap \{t>t_\init\})\) with the same
bound solving
\begin{equation*}
  \left\{
    \begin{aligned}
      &(X_0 - \oL_0^\ext) w = G + \sum_{i=1}^{m} X_i^t F^i
      & &\text{in } \Omega_t^\ext \cap \{t>t_\init\}, \\
      &w = 0
      & &\text{on } \{t_\init\} \times B^\ext \cup (t_\init,t_2) \times
      \partial B^\ext
    \end{aligned}
  \right.
\end{equation*}
Here the bound on
\(\| (1-\chi) \nabla w^\varepsilon \|_{\fsL^2(\Omega_t^\ext \cap
  \{t>t_\init\})}\) imply by the trace theorem that \(w=0\) on
\((t_\init,t_2) \times \partial B^\ext\) has a well-defined meaning
and still holds for the limit \(w^\epsilon\).

Going back to the equation, this shows on \(\Omega_t\) that
\begin{equation*}
  \| w \|_{\fsHyp^1(\Omega_t \cap \{t>t_\init\}}
  + \| X_0 w \|_{\fsHyp^{-1}(\Omega_t \cap \{t>t_\init\}}
  \lesssim
  \| F \|_{\fsL^2(\Omega_t^\ext \cap \{t>t_\init\})}
  +
  \| G \|_{\fsL^2(\Omega_t^\ext \cap \{t>t_\init\})}.
\end{equation*}

Under the commutator condition,
\textcite{hoermander-1967-hypoelliptic} shows for some \(s>0\) that
\(\| w \|_{\fsH^s(\Omega_t \cap \{t>t_\init\}} \lesssim \| w
\|_{\fsHyp^1(\Omega_t \cap \{t>t_\init\}} + \| X_0 w
\|_{\fsHyp^{-1}(\Omega_t \cap \{t>t_\init\}}\), which implies the
thought bound \eqref{eq:l0-estimate} by Sobolev embedding for some
\(p_1>2\) and \(\gamma_0=\gamma_1=2\).

\begin{remark}
  The discussion of local hypoelliptic operator to the whole space
  with uniform bounds is discussed in
  \textcite[Part~1]{bramanti-brandolini-lanconelli-uguzzoni-2010-non-hoermander}.
\end{remark}

% Biblatex
\AtNextBibliography{\small}
\printbibliography
% Old bibtex
% \bibliographystyle{hplain}
% \bibliography{lit}
\end{document}